\documentclass[a4paper,reqno,10pt]{amsart}

\usepackage{hyperref}
\usepackage{a4wide}

\usepackage{amssymb}
\usepackage{amstext}
\usepackage{amsmath}
\usepackage{amscd}
\usepackage{amsthm}
\usepackage{amsfonts}

\usepackage{graphicx}
\usepackage{latexsym}
\usepackage{mathrsfs}

\usepackage{enumitem}
\setlist[enumerate,1]{label={\upshape(\arabic*)}}
\setlist[enumerate,2]{label={\upshape(\alph*)}}

\usepackage{tikz}
\usetikzlibrary{cd,arrows,matrix,backgrounds,positioning,calc,decorations.markings,decorations.pathmorphing,decorations.pathreplacing}
\tikzset{black/.style={circle,fill=black,inner sep=3pt,outer sep=3pt},
         white/.style={circle,fill=white,draw=black,inner sep=3pt,outer sep=3pt},
}

\usepackage{array}
\newcolumntype{C}{>{$}c<{$}}

\usepackage{makecell}

\newtheorem{theorem}{Theorem}[section]

\newtheorem{corollary}[theorem]{Corollary}
\newtheorem{lemma}[theorem]{Lemma}
\newtheorem*{lemma*}{Lemma}
\newtheorem*{theorem*}{Theorem}
\newtheorem{proposition}[theorem]{Proposition}
\newtheorem{definition-proposition}[theorem]{Definition-Proposition}

\theoremstyle{definition}
\newtheorem{definition}[theorem]{Definition}

\newtheorem{remark}[theorem]{Remark}
\newtheorem{example}[theorem]{Example}

\renewcommand{\AA}{\mathcal{A}}

\newcommand{\CC}{\mathcal{C}}

\newcommand{\DD}{\mathcal{D}}

\renewcommand{\SS}{\mathcal{S}}

\newcommand{\WW}{\mathcal{W}}

\newcommand{\End}{\operatorname{End}\nolimits}

\newcommand{\RHom}{\mathbf{R}\strut\kern-.2em\operatorname{Hom}\nolimits}

\newcommand{\Cokernel}{\operatorname{Coker}\nolimits}

\newcommand{\coker}{\Cokernel}

\newcommand{\ov}{\overline}

\DeclareMathOperator{\moduleCategory}{\mathsf{mod}} \renewcommand{\mod}{\moduleCategory}

\DeclareMathOperator{\simp}{\mathsf{sim}}

\DeclareMathOperator{\brick}{\mathsf{brick}}

\DeclareMathOperator{\Filt}{\mathsf{Filt}}

\newcommand{\infl}{\rightarrowtail}
\newcommand{\defl}{\twoheadrightarrow}

\newenvironment{sbmatrix}{\left[\begin{smallmatrix}}{\end{smallmatrix}\right]}

\newcommand{\EE}{\mathcal{E}}

\numberwithin{equation}{section}

\begin{document}
\title[Schur's lemma for exact categories implies abelian]{Schur's lemma for exact categories implies abelian}

\author[H. Enomoto]{Haruhisa Enomoto}
\address{Graduate School of Mathematics, Nagoya University, Chikusa-ku, Nagoya. 464-8602, Japan}
\email{m16009t@math.nagoya-u.ac.jp}
\begin{abstract}
  We show that for a given exact category, there exists a bijection between semibricks (pairwise Hom-orthogonal set of bricks) and length wide subcategories (exact extension-closed length abelian subcategories). In particular, we show that a length exact category is abelian if and only if simple objects form a semibrick, that is, the Schur's lemma holds.
\end{abstract}

\maketitle

\section{Introduction}
In this paper, we consider the classical Schur's lemma in the context of Quillen's exact categories, from the viewpoint of \emph{semibricks} and \emph{wide subcategories}.

The Schur's lemma describes possible morphisms between simple objects in an abelian category, and is very fundamental and has lots of consequences. The Schur's lemma can be rephrased as follows: the set of simple objects in an abelian category forms a \emph{semibrick}, that is, every morphism is either zero or an isomorphism.

We can generalize the notion of simple objects in an abelian category to an exact category $\EE$, and has been investigated by several papers such as \cite{eno,bhlr,hr}. It is natural to ask whether the Schur's lemma holds in $\EE$, that is, simples in $\EE$ form a semibrick.
We will see in Corollary \ref{cor:schur} that, under mild assumption, this implies that $\EE$ is abelian. In other words, the Schur's lemma characterizes abelian categories among exact categories.

We will deduce this from the correspondence between semibricks and \emph{wide subcategories} of $\EE$. A wide subcategory is an extension-closed exact abelian subcategory of $\EE$, and has been investigated in the representation theory of a finite-dimensional algebra $\Lambda$ for the case $\EE = \mod\Lambda$, e.g. \cite{asai,it,ms}. In particular, wide subcategories of $\mod\Lambda$ are in bijection with other important objects, such as $\tau$-tilting modules, torsion(-free) classes, and so on (\cite{asai}).

For a wide subcategory $\WW$ of $\EE$, the Schur's lemma implies that the simples in $\WW$ is a semibrick. Conversely, it was shown by Ringel \cite{ringel} that for a given semibrick $\SS$ in an abelian category $\AA$, we can form a wide subcategory of $\AA$ whose simples are precisely objects in $\SS$. This immediately gives a bijection between semibricks in $\AA$ and a length wide category of $\AA$, and the aim of this paper is to generalize this bijection to exact categories (Theorem \ref{thm:main}).

\subsection{Conventions and notation}
Throughout this paper, \emph{all subcategories are assumed to be full and closed under isomorphisms}.
As for exact categories, we use the terminologies \emph{inflations}, \emph{deflations} and \emph{conflations}.
We refer the reader to \cite{buhler} for the basics of exact categories. We often denote by $0 \to X \to Y \to Z \to 0$ a conflation, $X \infl Y$ an inflation and $Y \defl Z$ a deflation.
A commutative diagram is \emph{exact} if every sub-diagram of the form $0 \to X \to Y \to Z \to 0$ is a conflation. For an inflation $X \infl Y$, we often denote $Y/X$ the cokernel of it.
Unless otherwise stated, we regard an abelian category as an exact category whose conflations are precisely usual short exact sequences, which we call the \emph{standard exact structure}. Also for an extension-closed subcategory $\DD$ of an exact category $\EE$, we always endow $\DD$ with the natural exact structure, that is, conflations in $\DD$ are precisely those in $\EE$ whose all terms are in $\DD$.

\section{Basic definitions and the main result}
First we introduce a \emph{semibrick} in an additive category.
\begin{definition}
  Let $\EE$ be an additive category.
  \begin{enumerate}
    \item An object $S$ in $\EE$ is a \emph{brick} if $\End_\EE(S)$ is a division ring, that is, $S$ is non-zero and every non-zero map $S \to S$ is an isomorphism.
    \item We denote by $\brick\EE$ the collection of isomorphism classes of all bricks in $\EE$.
    \item A subset $\SS$ of $\brick\EE$ is a \emph{semibrick} if $\EE(S,T) = 0$ holds for every two distinct objects $S$ and $T$ in $\SS$.
  \end{enumerate}
\end{definition}
Roughly speaking, a semibrick is a set of objects which satisfy the Schur's lemma like property.
Note that $\brick\EE$ may not form a set in general, but we require that a semibrick is actually a \emph{set} of bricks.

Next we introduce \emph{simple objects} in exact categories.
\begin{definition}
  Let $\EE$ be an exact category.
  \begin{itemize}
    \item An object $M$ in $\EE$ is called a \emph{simple object} in $\EE$ if $M$ is non-zero and there is no conflation $0 \to L \to M \to N \to 0$ in $\EE$ with $L,N \neq 0$.
    \item We denote by $\simp\EE$ the collection of isomorphism classes of simple objects in $\EE$.
    \item For a collection $\CC$ of objects in $\EE$, we denote by $\Filt\CC$ the subcategory of $\EE$ consisting of $X \in \EE$ such that there is a chain $0 = X_0 \infl X_1 \infl \cdots \infl X_l = X$ of inflations satisfying $X_i/X_{i-1} \in \CC$ for each $i$. We call such a chain \emph{$\CC$-filtration} and $l$ a \emph{length} of this $\CC$-filtration.
  \end{itemize}
\end{definition}
We say that an exact category $\EE$ is \emph{length} if $\simp\EE$ is a set and $\EE = \Filt(\simp\EE)$ holds, that is, every object has a $(\simp\EE)$-filtration.
For example, any extension-closed subcategory of an length abelian category is a length exact category.

A typical example of semibricks is the set of simple objects in an abelian category. The proof is the same as the classical Schur's lemma for modules.
\begin{lemma}[Schur's lemma]\label{lem:Schur}
  Let $\AA$ be an abelian category. Then the following hold.
  \begin{enumerate}
    \item Every simple object in $\AA$ is a brick.
    \item For two simple objects $S$ and $T$ in $\AA$, if $S$ and $T$ are not isomorphic, then $\AA(S,T) = 0$.
  \end{enumerate}
  In particular, if $\simp\AA$ is a set (e.g. if $\AA$ is length), then $\simp\AA$ is a semibrick.
\end{lemma}

Then we introduce a \emph{wide subcategory} of exact categories.
\begin{definition}
  Let $\EE$ be an exact category. A subcategory $\WW$ of $\EE$ is a \emph{wide} if the following conditions are satisfied:
  \begin{enumerate}
    \item $\WW$ is closed under extensions, that is, for any conflation $0 \to L \to M \to N \to 0$ in $\EE$, if $L$ and $N$ belong to $\EE$, then so does $M$.
    \item $\WW$ is an abelian category.
    \item The inclusion functor $\AA \hookrightarrow \EE$ is exact, that is, every usual short exact sequence $0 \to L \to M \to N \to 0$ in $\AA$ is a conflation in $\EE$.
  \end{enumerate}
  If in addition $\AA$ is a length abelian category, we say that $\WW$ is a \emph{length wide subcategory} of $\EE$.
\end{definition}

Now we can state the main result in this paper.
\begin{theorem}\label{thm:main}
  Let $\EE$ be an exact category. Then assignments $\SS \mapsto \Filt\SS$ and $\WW \mapsto \simp\WW$ give one-to-one correspondence between the following two classes.
  \begin{enumerate}
    \item The class of semibricks $\SS$ in $\EE$.
    \item The class of length wide subcategories $\WW$ in $\EE$.
  \end{enumerate}
\end{theorem}

\begin{remark}
  This is a generalization of the classical result of Ringel \cite[1.2]{ringel}, where $\EE$ was assumed to be an abelian category with the standard exact structure. In fact, if $\EE$ is realized as an extension-closed subcategory of an abelian category, then this theorem can be deduced from the the Ringel's result. However, we give a proof which does not use any Gabriel-Quillen type embedding of $\EE$ into an abelian category, and thus reprove the Ringel's result.
\end{remark}

We will give a proof in the next section. Before this, let us observe some consequences.
First of all, this theorem is a bit surprising since the notion of semibricks does not depend on the exact structure on $\EE$, while the notion of wide subcategories clearly does. Indeed, if we consider two different exact structures, then it may happen that $\Filt\SS$ may differ for the same semibrick $\SS$ in $\EE$, as the following example shows.
\begin{example}
  Let $\Lambda$ be an artinian ring and $\mod\Lambda$ the category of finitely generated $\Lambda$-modules. Then $\mod\Lambda$ is a length abelian category. Denote by $\SS$ the set of isomorphism classes of simple $\Lambda$-modules, which is a semibrick by Schur's lemma.
  First we endow $\mod\Lambda$ with the standard exact structure. Then clearly $\Filt\SS = \mod\Lambda$ holds.
  On the other hand, we can endow $\mod\Lambda$ with the split exact structure, that is, conflations are only split short exact sequences. Then in this exact structure, $\Filt\SS$ is the category of finitely generated semisimple $\Lambda$-modules, which is not equal to $\mod\Lambda$ unless $\Lambda$ is semisimple. This category is not closed under extension in the standard exact structure on $\mod\Lambda$, although it is an exact abelian subcategory of $\mod\Lambda$.
\end{example}

Another application is a characterization of abelian categories via the Schur's lemma.
\begin{corollary}\label{cor:schur}
  Let $\EE$ be a length exact category. Then the following are equivalent.
  \begin{enumerate}
    \item $\EE$ is an abelian category with the standard exact structure.
    \item $\simp\EE$ is a semibrick, that is, for every two simple objects $S$ and $T$ in $\EE$, every morphism $S \to T$ is either zero or an isomorphism.
  \end{enumerate}
\end{corollary}
\begin{proof}
  Clearly (1) implies (2) by the Schur's lemma. Conversely, suppose that $\simp\EE$ is a semibrick. Theorem \ref{thm:main} implies that $\Filt(\simp\EE)$ is a wide subcategory of $\EE$, and $\Filt(\simp\EE) = \EE$ holds since $\EE$ is length. Thus $\EE$ is a wide subcategory of $\EE$, which is clearly equivalent to (1).
\end{proof}
\begin{remark}
  In \cite[Lemma 3.6]{hr}, it was shown that every \emph{admissible} morphism between simples is either zero or an isomorphism. Since every morphism in an abelian category is admissible, it generalizes the classical Schur's lemma. However, in general there are many non-admissible morphisms in $\EE$ (actually, every morphism is admissible if and only if $\EE$ is abelian by Proposition \ref{prop:ab}). In particular, Corollary \ref{cor:schur} says that for a \emph{non-abelian} length exact category $\EE$, there always exist non-zero non-isomorphisms between simples, which is a bit surprising to the author.
\end{remark}

\section{Proof of the main theorem}
\emph{Throughout this section, we denote by $\EE$ an exact category.}
The following characterization of abelian categories among exact categories are useful. Since this is well-known (e.g. \cite[Exercise 8.6]{buhler}) and the proof is easy, we omit it.
\begin{proposition}\label{prop:ab}
  Let $\EE$ be an exact category. Then the following are equivalent:
  \begin{enumerate}
    \item $\EE$ is an abelian category with the standard exact structure.
    \item Every morphism $\varphi$ in $\EE$ is admissible in the sense of \cite[Definition 8.1]{buhler}, that is, it can be written as $\varphi = \iota\pi$ such that $\pi$ is a deflation and $\iota$ is an inflation.
  \end{enumerate}
\end{proposition}

By using this characterization, we can give the following criterion for wide subcategories.
\begin{lemma}\label{lem:widecri}
  Let $\WW$ be a subcategory of $\EE$. Then $\WW$ is a wide subcategory of $\EE$ if and only if the following conditions are satisfied:
    \begin{enumerate}[label={\upshape(\alph*)}]
      \item For any conflation $0 \to X \to Y \to Z \to 0$, if two out of $X,Y,Z$ belong to $\WW$, then so does the third.
      \item For any morphism $\varphi \colon X \to Y$ in $\WW$, there are an object $W$ in $\WW$, a deflation $\pi\colon X \defl W$ and an inflation $\iota\colon W \to Y$ satisfying $\varphi = \iota\pi$.
    \end{enumerate}
\end{lemma}
\begin{proof}
  First suppose that $\WW$ is a wide subcategory.

  (a)
  Let $0 \to X \xrightarrow{\varphi} Y \to Z \to 0$ be a conflation in $\EE$. Since $\WW$ is closed under extensions, if $X$ and $Z$ belong to $\WW$, then so does $Y$. If $X$ and $Y$ belongs to $\EE$, then we obtain a usual short exact sequence $0 \to X \xrightarrow{\varphi} Y \to \coker \varphi \to 0$ in an abelian category $\AA$.
  Since the embedding $\AA \hookrightarrow \EE$ is exact, this is actually a conflation in $\EE$. It follows that $Z$ is isomorphic to $\coker\varphi$, thus $Z \in \AA$. The remaining case is similar.

  (b)
  It is easy to see that the exact structure induced from the embedding $\AA \hookrightarrow \EE$ coincides with the standard exact structure on $\AA$. Thus (b) follows from Proposition \ref{prop:ab}.

  Conversely, suppose that $\WW$ satisfies (a) and (b). Then by (a), $\WW$ is closed under extensions in $\EE$, thus $\WW$ can be regarded as an exact category, and the embedding $\WW \hookrightarrow \EE$ is exact. By (a) and (b), we can check that this exact category $\WW$ satisfies Proposition \ref{prop:ab} (2), thus $\WW$ is an abelian category with the usual exact structure. Therefore $\WW$ is a wide subcategory of $\EE$.
\end{proof}

Next we prove the following basic properties on $\Filt\CC$.
\begin{lemma}\label{lem:filtext}
  Let $\CC$ be a collection of objects in $\EE$. Then the following hold.
  \begin{enumerate}
    \item $\Filt \CC$ is the smallest extension-closed subcategory of $\EE$ containing $\CC$.
    \item If $X$  has a $\CC$-filtration $0 = X_0 \infl X_1 \infl \cdots \infl X_l = X$, then $X_i$ and $X/X_i$ has $\CC$-filtrations of length $i$ and $l-i$ respectively for each $i$.
  \end{enumerate}
\end{lemma}
\begin{proof}
  This immediately follows from the Noether isomorphism theorem in exact categories, see \cite[Proposition 2.5]{eno} for example.
\end{proof}

Now we will prove that $\Filt\SS$ is wide for a semibrick $\SS$. The following is a key lemma.
\begin{lemma}\label{lem:frombrick}
  Let $\SS$ be a semibrick in $\EE$, and let $\varphi \colon S \to Y$ be a morphism in $\EE$ with $S \in \SS$ and $Y \in \Filt \SS$. Then either $\varphi$ is zero or $\varphi$ is an inflation in $\EE$ satisfying $Y/S \in \Filt\SS$.
\end{lemma}
\begin{proof}
  Assume that $\varphi$ is non-zero.
  Take an $\SS$-filtration $0 = Y_0 \infl Y_1 \infl \cdots \infl Y_{m-1} \infl Y_m = Y$ of $Y$. Let $\iota_i$ and $\pi_i$ be the following natural morphisms which give the following conflation for $i = 0, \dots,m-1$:
  \[
  \begin{tikzcd}
    0 \rar & Y_i \rar["\iota_i"] & Y_{i+1} \rar["\pi_{i+1}"] & Y_{i+1}/Y_i \rar & 0
  \end{tikzcd}
  \]
  Consider the composition $\pi_m\varphi\colon S \to Y/Y_{m-1}$. Since $S$ and $Y/Y_{m-1}$ belong to the semibrick $\SS$, we have that $\pi_m \varphi$ is either zero or an isomorphism. If $\pi_m\varphi = 0$, then $\varphi$ factors through $\iota_{m-1}$, that is, we have a map $\varphi_1 \colon S \to Y_{m-1}$ which makes the following diagram commutes:
  \[
  \begin{tikzcd}
    & & S \dar["\varphi"] \ar[ld, dashed, "\varphi_1"']\\
    0 \rar & Y_{m-1} \rar["\iota_{m-1}"] & Y \rar["\pi_m"] & Y/Y_{m-1} \rar & 0
  \end{tikzcd}
  \]
  Since $\varphi \neq 0$, by repeating this process, we obtain a map $\varphi_i \colon S \to Y_{m-i}$ such that $\pi_{m-i}\varphi_i$ is an isomorphism and that $\varphi = \iota_{m-1} \cdots \iota_{m-i}\varphi_i$. Then we have the following commutative diagram:
  \[
  \begin{tikzcd}
    & & S \dar["\varphi_i"] \ar[rd, "\pi_{m-i}\varphi_i"]\\
    0 \rar & Y_{m-i-1} \rar["\iota_{m-i-1}"] & Y_{m-i} \rar["\pi_{m-i}"] & Y_{m-i}/Y_{m-i-1} \rar & 0
  \end{tikzcd}
  \]
  It follows that $\pi_{m-i}$ is a retraction, thus the above conflation splits. Therefore, the above diagram is isomorphic to the following commutative diagram:
  \[
  \begin{tikzcd}
    & & S \dar["\begin{sbmatrix}a \\ 1 \end{sbmatrix}"'] \ar[rd, "1"]\\
    0 \rar & Y_{m-i-1} \rar["\begin{sbmatrix}1 \\ 0 \end{sbmatrix}"] & Y_{m-i-1}\oplus S \rar["{[0,1]}"'] & S \rar & 0
  \end{tikzcd}
  \]
  Here $1$ denotes the identity maps and $a$ denotes some map. Since $^t [a,1]$ is isomorphic to $^t [0,1]\colon S \to Y_{m-i-1} \oplus S$ as morphisms, so is $\varphi_i$. Moreover, since $^t[0,1]$ is an inflation with its cokernel $Y_{m-i-1}$, so is $\varphi_i$. Since inflations are closed under compositions, $\varphi = \iota_{m-1} \cdots \iota_{m-i} \varphi_i$ is an inflation.
  Now we can form the following exact commutative diagram, see \cite[Lemma 3.5]{buhler}.
  \[
  \begin{tikzcd}
    & & 0 \dar & 0 \dar \\
    0 \rar & S \rar["\varphi_i"]\dar[equal] & Y_{m-i} \rar\dar["\iota"] & Y_{m-i-1} \rar\dar & 0 \\
    0 \rar & S \rar["\varphi"] & Y \rar\dar & Y/S \rar\dar & 0 \\
    & & Y/Y_{m-i} \dar\rar[equal] & Y/Y_{m-i}\dar \\
    & & 0 & 0
  \end{tikzcd}
  \]
  Here $\iota:=\iota_{m-1} \cdots \iota_{m-i}$. Then Lemma \ref{lem:filtext} implies that $Y_{m-i-1}$ and $Y/Y_m$ belong to $\Filt\SS$, and so does $Y/S$.
\end{proof}
Now we can show that $\Filt\SS$ is wide.
\begin{lemma}\label{lem:semiwide}
  Let $\SS$ be a semibrick in $\EE$. Then $\Filt\SS$ is a wide subcategory of $\EE$.
\end{lemma}
\begin{proof}
  We check two conditions (a) and (b) in Lemma \ref{lem:widecri}. Note that $\Filt\SS$ is extension-closed by Lemma \ref{lem:filtext}.
  Let $\varphi \colon X \to Y$ be a morphism in $\Filt\SS$. By duality, to prove (a), it suffices to show that (a$'$) if $\varphi$ is an inflation, then $Y/X$ belongs to $\Filt\SS$. To prove (b), we show that there is a factorization $\varphi = \iota\pi$ inside $\Filt\SS$ such that $\pi$ is a deflation and $\iota$ is an inflation.

  We prove both (a$'$) and (b) by induction on the length $l$ of an $\SS$-filtration of $X$. The case $l=0$ is trivial, and the case $l=1$ immediately follows from Lemma \ref{lem:frombrick}.

  Suppose $l > 1$. Then by Lemma \ref{lem:filtext}, we can take a conflation $0 \to S \to X \to X/S \to 0$ such that $S \in \SS$ and $X/S$ has an $\SS$-filtration of length $l-1$.

  (a$'$)
  Suppose that $\varphi \colon X \to Y$ is an inflation. Then we can form the following exact commutative diagram, see \cite[Lemma 3.5]{buhler}.
  \[
  \begin{tikzcd}
    & & 0 \dar & 0\dar \\
    0 \rar & S \dar[equal]\rar & X \rar\dar & X/S\dar \rar & 0 \\
    0 \rar & S \rar & Y \rar\dar & Y/S \dar \rar & 0 \\
    &  & Y/X \rar[equal] \dar& Y/X \dar \\
    & & 0 & 0
  \end{tikzcd}
  \]
  Lemma \ref{lem:frombrick} implies that $Y/S \in \Filt\SS$ holds. Thus we can apply the induction hypothesis to the rightmost column to conclude $Y/X \in \Filt\SS$.

  (b)
  Consider the following diagram:
  \[
  \begin{tikzcd}
    0 \rar & S \rar["f"] & X \rar["g"]\dar["\varphi"] & X/S\rar \ar[ld, dashed, "\psi"] & 0 \\
  & &Y \\
  \end{tikzcd}
  \]
  If $\varphi f$ is zero, then there is a map $\psi \colon X/S \to Y$ which make the above diagram commutes. Then by the induction hypothesis, we can write as $\psi = i p$ inside $\Filt\SS$ with $p$ a deflation and $i$ an inflation. Thus $\varphi = i (p g)$ gives the desired factorization.

  Suppose that $\varphi f$ is non-zero. Then $\varphi f \colon S \to Y$ is an inflation with $Y/S \in \Filt\SS$ by Lemma \ref{lem:frombrick}.
  Thus we obtain the following exact commutative diagram.
  \[
  \begin{tikzcd}
    0 \rar & S \dar[equal]\rar & X \ar[rd, phantom, "{\rm p.b.}"] \rar\dar["\varphi"] & X/S\dar["\ov{\varphi}"] \rar & 0 \\
    0 \rar & S \rar & Y \rar & Y/S  \rar & 0 \\
  \end{tikzcd}
  \]
  Then the right square is a pullback diagram, see \cite[Proposition 2.12]{buhler}. By induction hypothesis, we obtain a commutative diagram
  \[
  \begin{tikzcd}
    X/S \ar[rr, "\ov{\varphi}"]\ar[rd, twoheadrightarrow, "p"'] & & Y/S \\
    & W \ar[ru, rightarrowtail, "i"']
  \end{tikzcd}
  \]
  with $W\in\Filt\SS$ and $p$ a deflation and $i$ an inflation. By taking the pullback, we obtain the following exact commutative diagram.
  \[
  \begin{tikzcd}
    0 \rar & S \rar \dar[equal]& W' \rar\dar["\iota"]\ar[rd, phantom, "{\rm p.b.}"] & W \dar["i", rightarrowtail] \rar & 0 \\
    0 \rar & S \rar & Y \rar & Y/S \rar & 0
  \end{tikzcd}
  \]
  The top row implies that $W'$ belongs to $\Filt\SS$ by Lemma \ref{lem:filtext}. Moreover, $\iota$ is an inflation, see \cite[Proposition 2.15]{buhler}. On the other hand, by considering the pullback of $W' \to W$ along $p$ (which exists since $p$ is a deflation), the universal property of the pullback square shows that the commutative exists.
  \[
  \begin{tikzcd}
    X \rar\dar["\pi", twoheadrightarrow]\ar[rd, phantom, "{\rm p.b.}"] \ar[dd,bend right, "\varphi"']& X/S \dar[twoheadrightarrow, "p"] \\
    W' \rar\dar[rightarrowtail, "\iota"] \ar[rd,phantom, "{\rm p.b.}"] & W\dar[rightarrowtail, "i"] \\
    Y \rar[twoheadrightarrow] & Y/S
  \end{tikzcd}
  \]
  Here $\pi$ is a deflation since it is a pullback of the deflation $p$. This proves (b).
\end{proof}

\begin{lemma}\label{lem:widewd}
  Let $\SS$ be a semibrick in $\EE$. Then $\Filt\SS$ is a length wide subcategory of $\EE$, and $\simp(\Filt\SS) = \SS$ holds.
\end{lemma}
\begin{proof}
  By Lemma \ref{lem:semiwide}, we have that $\Filt\SS$ is a wide subcategory of $\EE$.
  Moreover, Lemma \ref{lem:frombrick} immediately implies that every $S$ in $\SS$ is a simple object in an abelian category $\Filt\SS$. Conversely, let $X$ be a simple object in $\Filt\SS$. Since $X$ is simple, it is clear that this $\SS$-filtration has length one, that is, $X \in \SS$.
  Thus $\simp(\Filt\SS) = \SS$ holds (in particular, $\simp(\Filt\SS)$ is a set).
  Since $\Filt\SS = \Filt(\simp(\Filt\SS))$ holds, $\Filt\SS$ is a length abelian category.
\end{proof}

\begin{proof}[Proof of Theorem \ref{thm:main}]
  For a length wide subcategory $\WW$ of $\EE$, Schur's Lemma \ref{lem:Schur} shows that $\simp\WW$ is a semibrick in $\EE$. For a semibrick $\SS$ of $\EE$, we have that $\Filt\SS$ is a wide subcategory of $\EE$ by Lemma \ref{lem:semiwide}.
  Thus the ``maps" $\Filt(-) \colon \{ \text{semibricks in $\EE$} \} \leftrightarrows \{\text{wide subcategories of $\EE$} \} \colon \simp(-)$ is well-defined.

  We will see that these maps are mutually inverse.
  For every semibrick $\SS$ in $\EE$, we have $\SS = \simp(\Filt\SS)$ by Lemma \ref{lem:widewd}.
  For the opposite direction, let $\WW$ be a length wide subcategory of $\EE$.
  Since $\WW$ is length, clearly we have $\WW = \Filt_\WW(\simp\WW)$, where $\Filt_\WW$ is $\Filt$ inside the abelian category $\WW$. Since $\WW$ is wide, the inclusion $\WW \hookrightarrow \EE$ is exact, which implies that $\WW \subset \Filt(\simp\WW)$ holds, where $\Filt$ is considered inside $\EE$.
  On the other hand, since $\WW$ is extension-closed, $\Filt(\simp\WW) \subset \WW$ holds. Thus $\WW = \Filt(\simp\WW)$ holds.
\end{proof}

\medskip\noindent
{\bf Acknowledgement.}
This work is supported by JSPS KAKENHI Grant Number JP18J21556.

\end{document}